\documentclass{LMCS}

\def\dOi{9(3:27)2013}
\lmcsheading%
{\dOi}
{1--14}
{}
{}
{Dec.~11, 2012}
{Sep.~25, 2013}
{}

\usepackage{
	amsmath, 
	amssymb, 
	amsthm, 
	color,
	graphicx,
	hyperref,
	stmaryrd 
}
\usepackage[backrefs]{amsrefs}
\usepackage[all]{xy} 

\newcommand{\addtheorem}[2]{
	\newtheorem{#1}[thm]{#2}
}
\addtheorem{con}{Conjecture}
\addtheorem{df}{Definition}
\addtheorem{pro}{Proposition}
\addtheorem{que}{Question}

\newcommand{\abs}[1]{\left\lvert#1\right\rvert}
\newcommand{\ch}{\chi}
\newcommand{\eps}{\varepsilon}
\newcommand{\la}{\langle}
\newcommand{\ra}{\rangle}
\newcommand{\restrict}{\upharpoonright}

\newcommand{\mc}{\mathcal}

\newcommand{\mf}{\mathfrak}
\newcommand{\nil}{\varnothing}
\newcommand{\of}[1]{\llbracket#1\rrbracket}

\newcommand{\E}{\mathbb{E}}
\newcommand{\Iff}{\Longleftrightarrow}

\renewcommand{\to}{\rightarrow}

\renewcommand{\P}{\mathbb{P}}

\DeclareMathOperator{\card}{card}

\DeclareMathOperator{\MLR}{MLR}
\DeclareMathOperator{\DNC}{DNC}

\DeclareMathOperator{\SD}{SD}		
		
\DeclareMathOperator{\IM}{IM}
\DeclareMathOperator{\CIM}{CIM}

\DeclareMathOperator{\W3R}{W3R}
\DeclareMathOperator{\MWC}{MWC}
\DeclareMathOperator{\SBI}{SBI}

\DeclareMathOperator*{\argmin}{arg\, min}
\DeclareMathOperator*{\limstar}{\hbox{$\lim^*$\!}}

\begin{document}
	\title[Asymptotic Hamming distance]{
		Randomness extraction and asymptotic\\Hamming distance
	}
	\author[C.~Freer]{Cameron E. Freer\rsuper a}
	\address{{\lsuper a}Computer Science and Artificial Intelligence Laboratory, Massachusetts Institute of Technology}
	\email{freer@math.mit.edu}
	\thanks{{\lsuper a}Work on this publication was made possible through the support of a grant from the John Templeton Foundation. The opinions expressed in this publication are those of the authors and do not necessarily reflect the views of the John Templeton Foundation.}
	\author[B.~Kjos-Hanssen]{Bj\o rn Kjos-Hanssen\rsuper b}
	\address{{\lsuper b}Department of Mathematics, University of Hawai\textquoteleft i at M\=anoa}
	\email{bjoernkh@hawaii.edu}
	\thanks{{\lsuper b}This material is based upon work supported by the National Science Foundation under Grant No.\ 0901020.}
	
	
        \ACMCCS{[{\bf Theory of computation}]: Models of computation---Computability\,/\,Probabilistic computation; Computational complexity and cryptography---Problems, reductions and completeness;  Randomness, geometry and discrete structures---Pseudorandomness and derandomization}

	\begin{abstract}
		We obtain a non-implication result in the Medvedev degrees by studying sequences that are close to Martin-L\"of random in asymptotic Hamming distance. Our result is that the class of stochastically bi-immune sets is not Medvedev reducible to the class of sets having complex packing dimension 1.
	\end{abstract}

	\maketitle
	\section{Introduction}

	We are interested in the extent to which an infinite binary sequence $X$, or equivalently a set $X\subseteq\omega$,
	that is algorithmically random (Martin-L\"of random) remains useful as a randomness source after modifying some of the bits.
	Usefulness here means that some algorithm (extractor) can produce a Martin-L\"of random sequence from the result $Y$ of modifying $X$. For further motivation see Subsection \ref{subsec:relation} and Section \ref{sec:the-way}.

	A set that lies within a small Hamming distance of a random set may be viewed as 
	produced by an adaptive adversary corrupting or fixing some bits after looking at the original random set. 
	Similar problems in the finite setting have been studied going back to Ben-Or and Linial \cite{Ben-Or.Linial:89}. 

		If $A$ is a finite set and $\sigma,\tau\in \{0,1\}^{A}$, then the Hamming distance $d(\sigma,\tau)$ is given by 
		\[
			d(\sigma,\tau) = \abs{\{n:\sigma(n)\ne\tau(n)\}}.
		\]

		Let the collection of all infinite computable subsets of $\omega$ be denoted by $\mf C$. Let $p:\omega\to\omega$. 
		For $X, Y\in 2^{\omega}$ and $N\subseteq\omega$ we define a notion of proximity, or similarity, by
		\[
			X\sim_{p,N}Y\quad\Iff\quad (\exists n_{0})(\forall n\in N,\, n\ge n_{0})( d(X\restrict n,Y\restrict n)\le p(n)).
		\]

	We will study the effective dimension of sequences that are $\sim_{p, N}$
	to certain algorithmically random reals for suitably slow-growing functions
	$p$.

		We use the following notation for a kind of neighborhood around $X$.
		\[
			[X]_{p,N} = \{Y : Y\sim_{p,N} X\}.
		\]
		Moreover, for a collection $\mc A$ of subsets of $\omega$,
		\[
			[\mc A]_{p,N} = \bigcup \left\{ [X]_{p,N} : X \in \mc A\right \}.
		\]

	\paragraph{Turing functionals as random variables.}
		Since a random variable must be defined for all elements of the sample space,
		we consider a Turing functional $\Phi$ to be a map into
		\[
			\Omega: = 2^{<\omega}\cup 2^\omega.
		\]
		Setting the domain of $\Phi$ to also be $\Omega$ allows for composing maps. 
		Let
		\[
			\Lambda^X(n) = X(n).
		\]
		Thus $\Lambda:2^\omega\to 2^\omega$ is the identity Turing functional. 

		We define a probability measure $\lambda$ on $\Omega$ called \emph{Lebesgue (fair-coin) measure},
		whose $\sigma$-algebra of $\lambda$-measurable sets is
		\[
			\mf F=\{\mc S\subseteq\Omega: \mc S\cap 2^\omega\text{ is Lebesgue measurable}\},
		\]
		by letting $\lambda(\mc S)$ equal the fair-coin measure of $\mc S\cap 2^\omega$. Thus $\lambda(2^\omega) = 1$ and $\lambda(2^{<\omega}) = 0$, that is, 
		$\lambda$ is concentrated on the functions that are actually total. 

		The \emph{distribution} of $\Phi$ is the measure $\mc S\mapsto \lambda\{X:\Phi^X\in\mc S\}$,
		defined on $\mf F$.
		Thus the distribution of $\Lambda$ is $\lambda$.

			If $X\in 2^{\omega}$ then $X$ is called a real, a set, or a sequence depending on context.
			If $I\subseteq \omega$ then $X\restrict I$ denotes $X$, viewed as a function, restricted to the set $I$.
			We denote the cardinality of a finite set $A$ by $\abs{A}$.
			Regarding $X,Y$ as subsets of $\omega$ and letting $ + $ denote sum mod two, note that 
			$(X + Y)\cap n = \{k<n: X(k)\ne Y(k)\}$ and generally for a set $I\subseteq\omega$, 
			$(X + Y)\cap I = \{k\in I: X(k)\ne Y(k)\}$. 

			For an introduction to algorithmic randomness the reader may consult the recent books by Nies \cite{NiesBook} and Downey and Hirschfeldt \cite{MR2732288}. 
			Let $\MLR$ denote the set of Martin-L\"of random elements of $2^{\omega}$. 
			For a binary relation $R$ we use a set-theoretic notation for image,
			\[
				R\of{A} = \{y: (\exists x\in A)(\la x,y\ra\in R)\}.
			\]

			Let the \emph{use} $\varphi^{X}(n)$ be the largest number used in the computation of $\Phi^{X}(n)$. 
			We write 
			\[
				\Phi^{X}(n)\downarrow @s
			\]
			if $\Phi^{X}(n)$ halts by stage $s$, with use at most $s$; if this statement is false, we write $\Phi^{X}(n)\uparrow@s$. 
			We may assume that the running time of a Turing reduction is the same as the use, because 
			any $X$-computable upper bound on the use is a reasonable notion of use.

			For a set $\mathcal A\subseteq 2^\omega$, let 
			\[
				\text{Interior}_{p,N}(\mc A) = \{X: (\forall Y\sim_{p,N} X)(Y\in\mc A)\}
			\]
			\[
				\subseteq\text{Interior}_{*}(\mc A) = \{X: (\forall Y=^*X)(Y\in\mc A)\}\subseteq\mc A
			\]
			where $=^*$ denotes almost equality for all but finitely many inputs. It is easy to see that 
			\[
				\text{Interior}_{p,N}(\MLR)=\nil
			\]
			whenever $N\subseteq\omega$ and $p$ is unbounded.
			\begin{df}[Effective convergence]
				Let $\{a_{n}\}_{n\in\omega}$ be a sequence of real numbers.
				\begin{itemize}
					\item 
						\emph{$\{a_{n}\}_{n\in\omega}$ converges to $\infty$ effectively} if 
						there is a computable function $N$ such that for all $k$ and all $n\ge N(k)$, $a_n\ge k$. 
					\item 
						\emph{$\{a_{n}\}_{n\in\omega}$ converges to $0$ effectively} if 
						the sequence $\{a^{-1}_{n}\}_{n\in\omega}$ converges to $\infty$ effectively.
				\end{itemize}
			\end{df}

			\begin{df}
				For a sequence of real numbers $\{a_{n}\}_{n\in\omega}$, $\limstar_{n\to\infty}a_{n}$ is the real number to which $a_{n}$ converges effectively, if any; and is undefined if no such number exists.
			\end{df}
			As a kind of effective big-O notation, $p_n = \omega^{*}(q_n)$ means 
				$\limstar_{n\to\infty}q_n/p_n = 0$, i.e., $q_n/p_n$ goes to zero effectively.

			\paragraph{Central Limit Theorem.} Let $\mc N$ be the cumulative distribution function for a standard normal random variable; so
			\[
				\mc N(x) = \frac{1}{\sqrt{2\pi}}\int_{-\infty}^x e^{-t^2/2}\,dt.
			\]
			Let $\P$ denote fair-coin probability on $\Omega$. We may write 
			\[
				\P(\text{Event}) = \P(\{X: X\in\text{ Event}\}) = \lambda\{X: X\in\text{ Event}\}.
			\]

			We will make use of the following quantitative version of the central limit theorem.

			\begin{thm}[Berry-Ess\'een\footnote{See for example Durrett \cite{Durrett}.}]\label{Berry-Esseen}
				Let $\{X_n\}_{n\ge 1}$ be independent and identically distributed real-valued random variables with the expectations 
				$\E(X_n) = 0$, $\E(X_n^2) = \sigma^2$, and $\E(|X_n|^3) = \rho<\infty$. Then 
				there is a constant $d$ (with $.41\le d\le .71$) such that for all $x$ and $n$, 
				\[
					\left|\P\left(\frac{\sum_{i = 1}^n X_i}{\sigma\sqrt{n}}\le x\right) - \mc N(x)\right|\le \frac{d\rho}{\sigma^3\sqrt{n}}.
				\]
			\end{thm} 
			We are mostly interested in the case $X_{n} = X(n)-\frac12$, $X(n)\in\{0,1\}$, 
			for $X\in 2^{\omega}$ under $\lambda$, in which case $\sigma = 1/2$. 

	\subsection{New Medvedev degrees}
	Let $\le_{s}$ denote Medvedev (strong) reducibility and let $\le_{w}$ denote Muchnik (weak) reducibility. 
	A recent survey of the theory behind these reducibilities is Hinman \cite{Hinman}.

		\begin{df}[see, e.g., {\cite{JL}}]
	\label{BIdef}
			A set $X$ is \emph{immune} if for each $N\in\mf C$, $N\not\subseteq X$. 
			If $\omega\setminus X$ is immune then $X$ is \emph{co-immune}. 
			If $X$ is both immune and co-immune then $X$ is \emph{bi-immune}.
		\end{df}

		\begin{df}\label{SBIdef}
			A set $X$ is \emph{stochastically bi-immune} if for each set $N\in\mf C$, 
			$X\restrict N$ satisfies the strong law of large numbers, i.e.,
			\[
				\lim_{n\to\infty}\frac{\abs{X\cap N\cap n}}{\abs{N\cap n}} = \frac12.
			\]
		\end{df}

		\begin{df}\label{SDdef}
			Let $0\le\mf p<1$. A sequence $X\in 2^{\omega}$ is \emph{$\mf p$-stochastically dominated} if 
			for each $L\in\mf C$,
			\[
				\limsup_{n\to\infty}\frac{\abs{L\cap n}}{n}>0
				\quad\Longrightarrow\quad
				(\exists M\in\mf C)\quad M\subseteq L\quad\text{and}\quad
				\limsup_{n\to\infty} \frac{\abs{X\cap M\cap n}}{\abs{M\cap n}}\le\mf p.
			\]
			The class of stochastically dominated sequences is denoted $\SD = \SD_{\mf p}$. 
			If $\omega\setminus X\in \SD_{\mf p}$ then we write $X\in\SD^{\mf p}$ and say that 
			$X$ is stochastically dominating.
		\end{df}

		Let $\IM$ denote the set of immune sets, $\CIM$ the set of co-immune sets, and $\W3R$ the set of weakly 3-random sets. Let $K$ denote prefix-free Kolmogorov complexity.

	\begin{df}[see, e.g., {\cite[Ch.~13]{MR2732288}}]
	\label{dimdef}
	The \emph{effective Hausdorff dimension} of $A\in 2^{\omega}$ is
	\[
		\dim_H(A) =
			\liminf_{n\in\omega}
			\frac{K(A\restrict n)}{n}.
	\]
	The \emph{complex packing dimension} of $A\in 2^{\omega}$ is 
	\[
		\dim_{cp}(A) =
			\sup_{N\in\mf C}
				\inf_{n\in N}
					\frac{K(A\restrict n)}{n}.
	\]
	The \emph{effective packing dimension} of $A\in 2^{\omega}$ is
	\[
		\dim_p(A) = 
			\limsup_{n\in\omega}
			\frac{K(A\restrict n)}{n}.
	\]
	\end{df}

	\begin{pro}
	For all $A\in 2^{\omega}$,
	\[
		0 \le \dim_{H}(A) \le \dim_{cp}(A) \le \dim_{p}(A) \le 1.
	\]
	\end{pro}
	\begin{proof}
	The inequality $\dim_{H}(A)\le\dim_{cp}(A)$ uses the fact that each cofinite set $N\subseteq\omega$ is in $\mf C$. 
	The inequality $\dim_{cp}(A)\le\dim_{p}(A)$ uses the fact that 
	each $N\in\mf C$ is an infinite subset of $\omega$.
	\end{proof}

	By examining the complex packing dimension of reals that are $\sim_{p, N}$
	to a Martin-L\"of random real for $p$ growing more slowly than $n/(\log
	n)$, we will derive our main result, which states the existence, for each
	Turing reduction $\Phi$, of a set $Y$ of complex packing dimension $1$ for
	which $\Phi^Y$ is not stochastically bi-immune.

	\subsection{Relation of our results to other recent results.}\label{subsec:relation}
	Jockusch and Lewis \cite{JL} prove that the class of bi-immune sets is
	Medvedev reducible to the class of \emph{almost diagonally non-computable functions} DNC$^*$, i.e., functions $f$ such that $f(x)=\varphi_x(x)$ for at most finitely many $x$. Downey, Greenberg, Jockusch, and Milans \cite{MR2835294} show that DNC$_3$ (the class of DNC functions taking values in $\{0,1,2\}$) and hence also its superset DNC$^*$, is not Medvedev above the class of Kurtz random sets. We do not know whether the class of stochastically bi-immune sets is Medvedev reducible to the class of DNC$^*$ functions. We show in Theorem \ref{notsur} below that from a set of complex packing dimension 1 one cannot uniformly compute a stochastically bi-immune set; on the other hand, to compute a DNC$^*$ function from a set of complex packing dimension 1 one would apparently also need to know the witnessing set $N\in\mf C$.

	\begin{df}[see, e.g., {\cite[Def.~7.6.4]{NiesBook}}]
		\label{MWCdef}
	A sequence $X\in 2^{\omega}$ is Mises-Wald-Church $(\MWC)$ stochastic if no partial computable monotonic selection rule can select a biased subsequence of $X$, i.e., a subsequence where the relative frequencies of 0s and 1s do not converge to $1/2$. 
	\end{df}

	\begin{df}
	\label{BI2def}
	A sequence $X\in 2^{\omega}$ is $\mathrm{BI}^{\,2}$ (bi-immune for sets of size two) if there is no computable collection of disjoint finite sets of size 2 on which the set omits a certain pattern such as $01$.
	More precisely, $X$ is $\mathrm{BI}^{\,2}$ if for each computable disjoint collection $\{T_n:n\in\omega\}$ where each $T_n$ has cardinality two, say $T_n=\{s_n,t_n\}$ where $s_n<t_n$, and each $P\subseteq \{0,1\}$, there is an $n$ such $X(s_n)=P(0)$ and $X(t_n)=P(1)$.
	\end{df}

	Each von Mises-Wald-Church stochastic (MWC-stochastic) set is stochastically bi-immune. Our main theorem implies that a set of complex packing dimension 1 does not necessarily uniformly compute a MWC-stochastic set. This consequence is not really new with the present paper, however, because the fact that $\DNC_{3}$ is not Medvedev above BI$^2$ is implicit in Downey, Greenberg, Jockusch, and Milans \cite{MR2835294} as pointed out to us by Joe Miller. The situation is diagrammatically illustrated in Figure \ref{crazy}, with notation defined in Figures \ref{explain} and \ref{sane}. In the future we could hope to replace complex packing dimension by effective Hausdorff dimension in Theorem \ref{notsur}.

	\begin{figure}
		\[
			\xymatrix{
				&
				&
				&
				&
				\DNC_{2}
				\ar[dd]
				\ar@{=>}[dl] &
				& 
				\\
				&
				&
				&
				*+[F]{\text{MLR}}\ar[d]\ar[dl] &
				&
				&
				\\
				&
				&
				*+[F]{\text{KR}}\ar[d] &
				*+[F]{\text{MWC}}\ar[d]\ar[dl] &
				\DNC_{3}\ar@{=>}[d]\ar@{~>}[dll]&
				&
				\\
				&
				&
				*+[F]{\text{BI}^{2}}\ar[d]	&
				*+[F]{\text{SBI}}\ar[dl] &
				*+[F]{(H,1)}\ar[d]&
				&
				\\
				&
				&
				*+[F]{\text{BI}} &
				&
				*+[F]{(cp,1)}\ar@{~>}[ul]&
				&
				\\
			}		 
		\]
		\caption{
			Some Medvedev degrees.
			The fact that $(cp,1)$ is not Medvedev above SBI is Theorem \ref{notsur}.
		}
		\label{crazy}
	\end{figure}

	\begin{figure}
		\[
			\xymatrix{
				\ar[r] & \text{Included in} \\
				\ar@{~>}[r] & \text{Not Medvedev above} \\
				\ar@{=>}[r] & \text{Medvedev above} \\
			}
		\]
		\caption{Meaning of arrows.}
		\label{explain}
	\end{figure}

	\begin{figure}
		\begin{tabular}{|c|c|c|}
			\hline
			Abbreviation &
			Unabbreviation &
			Definition \\
			\hline
			$\DNC_{n}$ &
			Diagonally non-computable function in $n^{\omega}$ &
			\\
			\hline
			$\MLR$ &
			Martin-L\"of random	&
			\\
			KR &
			Kurtz random (weakly 1-random) &
			\\
			\hline
			$\MWC$ &
			Mises-Wald-Church stochastic &
			\ref{MWCdef} \\
			SBI	&
			Stochastically bi-immune &
			\ref{SBIdef} \\
			BI &
			bi-immune &
			\ref{BIdef} \\
			BI$^2$ &
			bi-immune for sets of size two &
			\ref{BI2def} \\
			\hline
			$(H,1)$	&
			effective Hausdorff dimension 1	&
			\ref{dimdef} \\
			$(cp,1)$ &
			complex packing dimension 1	&
			\ref{dimdef} \\
			\hline
		\end{tabular}
		\caption{Abbreviations used in Figure \ref{crazy}.}
		\label{sane}
	\end{figure}

	\section{Hamming space}

		The Hamming distance between a point and a set of points is defined by $d(y,A): = \min_{a\in A}d(y,a)$.
		The $r$-neighborhood of a set $A\subseteq\{0,1\}^{n}$ is 
		\[
			\Gamma_{r}(A) = \{y\in\{0,1\}^{n}:d(y,A)\le r\}.
		\]
		In particular,
		\[
			\Gamma_{r}(\{c\}) = \{y\in\{0,1\}^{n}: d(y,c)\le r\},
		\]
		and
		\[
			\Gamma_{r}(A) = \bigcup_{a\in A}\Gamma_{r}(\{a\}).
		\]
		A \emph{Hamming-sphere}\footnote{A Hamming-sphere is more like a ball than a sphere, but the terminology is entrenched.} 
		with center $c\in\{0,1\}^{n}$ is a set $S\subseteq\{0,1\}^{n}$ such that for some $k$,
		\[
			\Gamma_{k}(\{c\})\subseteq S\subseteq\Gamma_{k + 1}(\{c\}).
		\]

		\begin{thm}[Harper \cite{Harper}; see also Frankl and F\"uredi \cite{FF}]\label{HarperThm}
		For each $n,r\ge 1$ and each set $A\subseteq\{0,1\}^{n}$, 
		there is a Hamming-sphere $S\subseteq\{0,1\}^{n}$ such that
		\[
			\abs{A} = \abs{S},\quad\text{and}\quad\abs{\Gamma_{r}(A)} \ge \abs{\Gamma_{r}(S)}.
		\]
		\end{thm}

		\noindent Following Buhrman et al.\ \cite{Buhrman}, we write
		\[
			b(n,k): = {n\choose 0} + \cdots + {n\choose k}.
		\]
		Note that for all $c\in\{0,1\}^{n}$, $\abs{\Gamma_{k}(\{c\})} = b(n,k)$ 

		If the domain of $\sigma$ is an interval $I$ in $\omega$ rather than an initial segment of $\omega$, 
		we may emphasize $I$ by writing
		\[
			B^{I}_{r}(\sigma) = \Gamma_{r}(\{\sigma\}) = \{ \tau\in\{0,1\}^{I} : d(\sigma,\tau)\le r \}.
		\]
		$\P$ denotes the uniform distribution on $\{0,1\}^{I}$, so by definition 
		\[
			\P(E) = \frac{\abs{E}}{2^{\abs{I}}}.
		\]

		Recall that $D_{m}$ is the $m^{\text{th}}$ canonical finite set.
		The intuitive content of Lemma \ref{core} below is that a medium size set is unlikely to contain a random large ball. (Note that we do not assume the sets $I_{m}$ are disjoint.)

		\begin{lem}\label{core}
			Let $\ch\in\omega^{\omega}$. Suppose 
			\begin{equation}
				\limstar_{n\to\infty}\ch(n)/\sqrt{n} = \infty.
			\end{equation}
			Let $f\in\omega^{\omega}$ be a computable function. 
			Let $I_{m} = D_{f(m)}$ and $n_{m} = \abs{I_{m}}$.
			Suppose
			\begin{equation}
				\limstar_{m\to\infty}n_{m} = \infty.
			\end{equation}

			For each $m\in\omega$ let $E_{m}\subseteq\{0,1\}^{I_{m}}$.
			Suppose $\limsup_{m\to\infty}\P(E_{m})\le \mf p$ where $0<\mf p<1$ is computable.
			Writing 
			\[
				B_{\ch(n)}(X)\quad\text{for}\quad B^{I_{m}}_{\ch(n_{m})}(X\restrict I_{m}),
			\]
			we have
			\begin{equation}\label{blue}
			\limstar_{m\to\infty}\P(\{X:B_{\ch(n)}(X)\subseteq E_{m}\}) = 0.
			\end{equation}
			Moreover, for each $m_{0}\in\omega$ and computable $\mf q\in (\mf p,1)$ 
			there is a modulus of effective convergence in (\ref{blue}) that works 
			for all sets $\{E_{m}\}_{m\in\omega}$ such that for all $m\ge m_{0}$, $\P(E_{m})\le \mf q$.
		\end{lem}
		\proof
			Let $X\in 2^\omega$ be a random variable with $X =_d\Lambda$, and
			\[
				S^{(m)} = \sum_{i\in I_m} X(i).
			\]
			Let 
			\[
				f_{m}(x) = \P\left( \frac{ S^{ (m) }-n/2}{ \sqrt{n}/2 } \le x\right)
			\] 
			We have%
			\footnote{Indeed, let $Y_{i} = X_{i}-\E(X_{i})$ 
			where $\E(X_{i}) = \frac{1}{2}$ is the expected value of $X_{i}$, so $\E(Y_{i}) = 0$. 
			By the Berry-Ess\'een Theorem \ref{Berry-Esseen}, for all $x$ 
			\[
				\left|\P\left(\frac{\sum_{i\in I_{m}} Y_i}{\sigma\sqrt{n}}\le x\right) - \mc N(x)\right|\le \frac{d\rho}{\sigma^3\sqrt{n}} = \frac{d}{\sqrt{n}},
			\]
			where $\rho = 1/8 = \E(\abs{Y_{i}}^{3})$, and 
			$\sigma = 1/2$ is the standard deviation of $X_{i}$ (and $Y_{i}$).}
			\[
				\abs{f_{m}(x)-\mc N(x)}\le \frac{d}{\sqrt{n_{m}}}.
			\]
			Since $\limstar_{m\to\infty} n_{m} = \infty$, $\limstar\frac{d}{\sqrt{n_{m}}} = 0$. So
			\begin{eqnarray}\label{eu}
				\limstar_{m\to\infty}\sup_{x}\abs{f_{m}(x)-\mc N(x)} = 0.
			\end{eqnarray}
			Let $r = r_{m}$ be such that 
			\[
				b(n,r)\le \abs{E_{m}}< b(n,r + 1).
			\]
			Let
			\[
				a_{m} = \frac{r_{m} -\frac{n}{2}}{\sqrt{n}/2},
			\]
			and let
			\[
				b_{m} = a_{m} + \frac{1}{\sqrt{n}/2} - \frac{\ch(n)}{\sqrt{n}/2}.
			\]
			By (\ref{eu}), 
			\begin{eqnarray}\label{dos}
				\limstar_{m\to\infty} |f_{m}(b_{m})-\mc N(b_{m})| = 0.
			\end{eqnarray}
			We have
			\[
				\limsup_{m\to\infty}f_{m}(a_{m}) 
				= \limsup_{m\to\infty} \P\left(\frac{S^{(m)} - n/2}{\sqrt{n}/2} 
				\le \frac{r_{m} -\frac{n}{2}}{\sqrt{n}/2}\right) 
			\]
			\[
				 = \limsup_{m\to\infty} \P(S^{(m)} \le r_{m})
				 = \limsup_{m\to\infty} \frac{b(n,r)}{2^n}
				\le\limsup_{m\to\infty} \P(E_m)
				\le \mf p.
			\]
			Since $f_{m}\to\mc N$ uniformly, it follows that
			\[
				\limsup_{m\to\infty} \mc N(a_{m})\le \mf p,
			\]
			and so as $\mc N$ is strictly increasing, 
			\[
				\limsup_{m\to\infty}a_{m}\le \mc N^{-1}(\mf p)\quad( = 0\text{ if }\mf p = 1/2).
			\]
			Let $m_0$ be such that for all $m\ge m_0$,
			\[
				a_{m} + \frac{1}{\sqrt{n_m}/2}\le \mc N^{-1}(\mf p)+1.
			\]
			Since by assumption $\limstar_{n\to\infty}\ch(n)/\sqrt{n} = \infty$, we have that $b_{m}$ is the sum of 
			a term that goes effectively to $-\infty$, and 
			a term that after $m_0$ never goes above $\mc N^{-1}(\mf p) + 1$ again. Thus 
			\begin{eqnarray*}\label{b}
				\limstar_{m\to\infty} b_{m} = -\infty.
			\end{eqnarray*}
			It is this rate of convergence that is transformed in the rest of the proof. Now 
			\begin{eqnarray*}\label{uno}
				\limstar_{m\to\infty} \mc N(b_{m}) = 0.
			\end{eqnarray*}
			Hence by (\ref{dos}), $\limstar_{m\to\infty} f_{m}(b_{m}) = 0$.

			Let us write
			\[
				B_{t}(X): =  B^{I_{m}}_{t}(X\restrict I_{m}),
			\]
			considering $X\restrict I_{m}$ as a string of length $n$. 
			By Harper's Theorem \ref{HarperThm}, we have a Hamming sphere $H$ with 
			\[
				\abs{H} = \abs{\neg E_{m}}\quad\text{and}\quad
				\abs{\Gamma_{\ch(n)}(\neg E_{m})}
				\ge
				\abs{\Gamma_{\ch(n)}(H)}.
			\]
			Then
			\[
				\P(\{X:X\in\Gamma_{\ch(n)}(\neg E_{m})\}) 
				\ge
				\P(\{X:X\in\Gamma_{\ch(n)}(H)\}).
			\]
			Therefore 
			\[
				\P(\{X: X\not\in\Gamma_{\ch(n)}(\neg E_{m})\})
				\le
				\P(\{X: X\not\in\Gamma_{\ch(n)}H)\}).
			\]
			Let $\widehat H$ be the complement of $H$. 
			If the Hamming sphere $H$ is centered at $c\in\{0,1\}^{n}$ then clearly 
			$\widehat H$ is a Hamming sphere centered at $\overline c$, where $\overline c(k) = 1-c(k)$.
			Since
			\[
				\abs{\widehat H} = \abs{E_{m}} < b(n,r + 1),
			\] 
			we have $\widehat H\subset\Gamma_{r + 1}(\{\overline c\})$. So we have:
			\[
				\P(\{X: B_{\ch(n)}(X)\subseteq E_{m}\})\le \P(\{X:B_{\ch(n)}(X)\subseteq \widehat H\})
			\]
			\[
				 < \P(\{X:B_{\ch(n)}(X)\subseteq \Gamma_{r + 1}(\{\overline c\})\})
				 = \frac{b(n,r + 1-\ch(n))}{2^{n}} 
			\]
			\[	=  \P[S^{(m)} \le r + 1-\ch(n)]
	 =  \P\left[\frac{S^{(m)} - \frac{n}{2}}{\sqrt{n}/2} \le \frac{r + 1 - \frac{n}{2} -\ch(n)}{\sqrt{n}/2}\right] 
			\]
			\[
				 = f_{m}\left(a_{m} + \frac{1}{\sqrt{n}/2} - \frac{\ch(n)}{\sqrt{n}/2}\right) = f_{m}(b_{m}).
			\]

	\noindent Since we showed that $\limstar_{m\to\infty} f_{m}(b_{m}) = 0$, and since by
	assumption $\limstar_{m\to\infty} n_{m} = \infty$,
			\[
				\limstar_{m\to\infty}\P(\{X:B_{\ch(n)}(X)\subseteq E_{m}\}) = 0.\eqno{\qEd} 
			\]

	\section{Turing reductions that preserve randomness}\label{sec:the-way}

		The way we will obtain our main result Theorem \ref{notsur} is by proving essentially that for any ``randomness extractor'' Turing reduction, and any random input oracle, a small number of changes to the oracle will cause the extractor to fail to produce a random output. This would be much easier if we restricted attention to Turing reductions having disjoint uses on distinct inputs, since we would be working with independent random variables. Indeed, one can give an easy proof in that case, which we do not include here.
		The main technical achievement of the present paper is to be able to work with overlapping use sets; key in that respect is Lemma \ref{sixways} below. The number of changes to the random oracle that we need to  make is small enough that the modified oracle has complex packing dimension 1. We were not able to set up the construction so as to guarantee effective Hausdorff dimension 1 (or even greater than 0); this may be an avenue for future work.

		For a set of pairs $E$, we have the projections $E^x = \{y:(x,y)\in E\}$ and $E_y = \{x:(x,y)\in E\}$. 

		\begin{lem}\label{HA15}
			Let $\mu_{1}$ and $\mu_{2}$ be probability measures on sample spaces $\Omega_{1}$ and $\Omega_{2}$ and 
			let $E$ be a measurable subset of $\Omega_{1}\times\Omega_{2}$. 
			Suppose that $\eta$, $\alpha$, and $\delta$ are positive real numbers such that 
			\begin{eqnarray}
				\mu_{1}E_{y}>\eta \quad(\forall y\in\Omega_{2}), \quad \text{and}\\
				\mu_{1}\{x:\mu_{2}E^{x}\le\alpha\}\ge 1-\delta.
			\end{eqnarray}
			Then $\eta<\alpha + \delta$.
		\end{lem}
		\proof
			By Fubini's theorem, 
			\[
				\eta<\int_{\Omega_{2}} \mu_{1}(E_{y})d\mu_{2}(y) 
				= \iint_{\Omega_{1}\times\Omega_{2}} E(x,y)d\mu_{1}(x)d\mu_{2}(y) 
				= \int_{\Omega_{1}} \mu_{2}(E^{x})d\mu_{1}(x)
			\]
			\[
				\le \alpha \cdot \mu_{1} \{ x : \mu_{2}(E^{x})\le\alpha\}  +  1\cdot \mu_{1}\{x: \mu_{2}(E^{x})\ge\alpha\} 
				\le \alpha\cdot 1  +  1\cdot\delta.\eqno{\qEd}
			\]

		\begin{df}
			For a real $X$ and a string $\sigma$ of length $n$,
			\[
				(\sigma\searrow X)(n)  = 
				\begin{cases}
					\sigma(n) & \text{if $n<|\sigma|$,} \\
					X(n) & \text{otherwise,} 
				\end{cases}
			\]
			and
			\[
				(\sigma^{\frown} X)(n)  = 
				\begin{cases}
					\sigma(n) & \text{if $n<|\sigma|$,} \\
					X(n-|\sigma|) & \text{otherwise.} 
				\end{cases}
			\]
			Thinking of $\sigma$ and $X$ as functions we may write
			\[
				\sigma\searrow X = \sigma\cup (X\restrict \omega\backslash |\sigma|)
			\]
			and thinking in terms of concatenation we may write
			\[
				\sigma^{\frown}X = \sigma\, X.
			\]
		\end{df}
		\begin{lem}\label{sixways}
			Let $\Phi$ be a Turing reduction such that 
			\begin{equation}\label{percent}
				\lambda(\Phi^{-1}\of{\SD_{\mf p}}) = 1 
			\end{equation}
			and let $\Phi_{\sigma}^{X} = \Phi^{\sigma\searrow X}$. Then for any finite set $\Sigma\subseteq 2^{<\omega}$,
			\[
				(\forall\eps>0)(\forall i_{0})(\exists i>i_0)(\forall \sigma\in \Sigma)
			\]
			\[
				\P(\{X\mid \Phi_{\sigma}^{X}(i) = 1\})\le \mf p + \eps.
			\]
		\end{lem}\enlargethispage{2\baselineskip}
		\begin{proof}
			First note that for all $\sigma\in 2^{<\omega}$, $\lambda(\Phi_{\sigma}^{-1}\of{\SD_{\mf p}} = 1$ as well. 

			Suppose otherwise, and fix $\eps$, $i_{0}$ and $\Sigma$ such that 
			\[
				(\forall i>i_0)(\exists \sigma\in \Sigma)\quad \P(\Phi_\sigma(i) = 1)>\mf p + \eps.
			\]
			By density of the rationals in the reals we may assume $\eps$ is rational and hence computable. 
			Since there are infinitely many $i$ but only finitely many $\sigma$, it follows that there is some $\sigma$ such that
			\begin{eqnarray}
				(\exists^\infty k>i_{0})\quad\P(\Phi_{\sigma}(k) = 1)> \mf p + \eps \label{won} 
			\end{eqnarray}
			and in fact 
			\[
				\limsup \abs{\{k<n: \P(\Phi_{\sigma}(k) = 1)> \mf p + \eps\}}/n>0.
			\]
			Fix such a $\sigma$ and let $\Psi = \Phi_\sigma$. 
			Let $\{\ell_n\}_{n\in\omega}$ be infinitely many values of $k$ in (\ref{won}) listed in increasing order; 
			note that $L = \{\ell_{n}\}_{n\in\omega}$ may be chosen as a computable sequence. 

			For an as yet unspecified subsequence $\mc K = \{k_{n}\}_{n\in\omega}$, $\mc K\subseteq L$, let
			\begin{eqnarray}
				E = \{(X,n):\Psi^{X}(k_{n}) = 1\}. 
			\end{eqnarray}
			We obtain then also projections $E_n = \{X: \Psi^X(k_n) = 1\}$, $E^X = \{n:\Psi^X(k_{n}) = 1\}$.
			By (\ref{won}) we have for all $n\in\omega$,
			\begin{eqnarray}\label{111}
				\lambda E_n> \mf p + \eps. 
			\end{eqnarray}
			The fraction of events $E_{n}$ that occur in $N = \{0,\ldots,N-1\}$ for $X$ is denoted 
			\[
				e_{N}^X = \frac{\abs{E^X\cap {N}}}{N} 
			\]
			By assumption (\ref{percent}), 
			\[
				\lambda\left\{X: 
					(\exists\mc K\subseteq L) (\exists M)(\forall N\ge M)
					\left( e_{N}^X \le \mf p + \frac{\eps}{2}\right)
				\right\} = 1.
			\]
			Thus there is an $M$ and a $\mc K$ (using that $\mf C$ is countable) such that 
			\begin{eqnarray} 
				\lambda\left\{X: e_{M}^X\le \mf p + \frac{\eps}{2} \right\} 
				\ge\lambda\left\{X: (\forall N\ge M) e_{N}^X \le \mf p + \frac{\eps}{2}\right\}
				\ge 1-\frac{\eps}{3}.\label{222}
			\end{eqnarray}
			Let $\Omega_{1}$ be the unit interval $[0,1]$. 
			Let $\Omega_{2} = M = \{0,1,\ldots,M-1\}$. 
			Let $\mu_{1} = \lambda$. 
			Let $\mu_{2} = \card$ be the counting measure on the finite set $M = \{0,1,\ldots,M-1\}$, 
			so that for a finite set $A\subset M$, $\card(A)$ is the cardinality of $A$.%
			\footnote{%
				In this case,
				$
						\int \mu_{1}(E_{y})d\mu_{2}(y) 
					= 	\int \mu_{1}(E_{n})d\mu_{2}(n) 
					= 	\int \lambda(E_{n}) d\,\card(n)
					= 	\sum_{n\in\Omega_{2}} \lambda(E_{n}) \card(\{n\})
					= 	\sum_{n = 0}^{M-1}\lambda(E_{n})\cdot 1
				$.
			}
			Let $\eta = \mf p + \eps$, $\alpha = \mf p + \eps/2$, and $\delta = \eps/3$, and note that 
			$\eta>\alpha + \delta$. 
			By (\ref{111}), (\ref{222}) and Lemma \ref{HA15}, $\eta<\alpha + \delta$, a contradiction.
		\end{proof}\unskip
	\section{Extraction and Hamming distance}\unskip
		\begin{thm}\label{Main}
			Let $\mf p<1$ be computable. Let $p:\omega\to\omega$ be any computable function such that $p(n) = \omega^{*}(\sqrt{n})$.
			Let $\Phi$ be a Turing reduction. 
			There exists an $N\in\mf C$ and an almost sure event $\mc A$ such that 
			\[
				\mc A\cap\mathrm{Interior}_{p,N}(\Phi^{-1}\of{\mc A}) = \nil
			\]
		\end{thm}
		\begin{proof}
			Let 
			\[
				\mc A := \W3R \subset \MLR \subset \CIM \cap \IM \cap (\SD_{\mf p}\cup \SD^{\mf p}).
			\]
			We show
			\begin{enumerate}
				\item If $\lambda(\Phi^{-1}\of\SD_{\mf p}) = 1$, or $\lambda(\Phi^{-1}\llbracket\SD^{\mf p}\rrbracket) = 1$, then 
				\[
					\MLR\cap\text{Interior}_{p,N}(\Phi^{-1}\llbracket\CIM\rrbracket) = \nil,\quad\text{ or}
				\]
				\[
					\MLR\cap\text{Interior}_{p,N}(\Phi^{-1}\llbracket\IM\rrbracket) = \nil,\quad\text{ respectively.}
				\]
				\item Otherwise; then 
				\[
					\W3R\cap\text{Interior}_{*}(\Phi^{-1}\llbracket\SD_{\mf p}\rrbracket) = \nil
				\]
				and
				\[
					\W3R\cap\text{Interior}_{*}(\Phi^{-1}\llbracket\SD^{\mf p}\rrbracket) = \nil
				\]
			\end{enumerate}
			\noindent Proof of (2): If we are not in case (1) then $\lambda\{X\mid \Phi^{X}\in\SD_{\mf p}\}<1$, so by the 0-1 Law, 
			$\lambda\{X\mid (\forall Y = ^{*}X)(\Phi^{Y}\in\SD_{\mf p})\} = 0$. 
			This is (contained in) a $\Pi^{0}_{3}$ null class, so if $X\in\W3R$ then 
			$(\exists Y = ^{*}X)(\Phi^{Y}\not\in\SD_{\mf p})$ hence we are done.

			\noindent Proof of (1): By Lemma \ref{sixways},
			\[
				(\exists\mf p<1)(\forall\eps>0)(\forall n)(\forall i)(\exists i'>i)(\forall \sigma\in 2^{ = n})
			\]
			\begin{eqnarray}\label{q}
				\P(\{Z:\Phi^{\sigma\searrow Z}(i') = 1\})\le \mf p  +  \eps;
			\end{eqnarray}
			Since $\Phi$ is total for almost all oracles, it is clear that $i'$ is a computable function $f(k,n)$ of $\eps = 1/k$ and $n$. 
			Let $g:\omega\to\omega$ be the computable function with $\lim_{n\to\infty}g(n) = \infty$ given by $g(s) = 2s$. 
			Let $n_{0} = 0$ and $i_{0} = 0$. Assuming $s\ge 0$ and $n_{s}$ and $i_{s}$ have been defined, let
			\[
				i_{s + 1} = f(g(s),n_s),
			\]
			and let $n_{s + 1}$ be large enough that
			\begin{eqnarray}\label{united}
				(\forall\sigma\in 2^{ = n_{s}})\quad \lambda\{Z\mid \Phi^{\sigma \searrow Z}(i_{s + 1}) \uparrow@n_{s + 1}\}\le \frac{1}{2s},
			\end{eqnarray}
			\[
	\limstar_{s\to\infty}			\ \frac{p(n_{s + 1}-n_{s})}{\sqrt{n_{s + 1}-n_{s}}} =\infty,\quad\text{and}\quad\sum_{k = 0}^{s}p(n_{k + 1}-n_{k})\le p(n_{s + 1}).
			\]
			Note that since $i'>i$ in (\ref{q}), we have $i_{s + 1}>i_{s}$ and hence $R: = \{i_0, i_1, \ldots\}$ is a computable infinite set. We now have 
			\begin{eqnarray}\label{airlines}
				(\forall s)(\forall \sigma\in 2^{ = n_{s}})\quad\P(\{Z:\Phi^{\sigma\searrow Z}(i_{s + 1}) = 1\})\le \mf p  +  \frac{1}{2s}
			\end{eqnarray}
			so
			\begin{eqnarray}\label{qq}
				\P(\{Z:\Phi^{\sigma\searrow Z}(i_{s + 1})\downarrow = 1 @n_{s + 1}\})\le\mf p +  \frac{1}{2s}.
			\end{eqnarray}
			Note $[a,b) = b\backslash a$. 

			Let $X\in\MLR$. We aim to define $Y\sim_{p}X$ such that $\Phi^{Y}\not\in\MLR$. 
			We will in fact make $Y\le_{T}X$, so we define a reduction $\Xi$ and let $Y = \Xi^{X}$. 
			Since we are defining $Y$ by modifying bits of $X$, the use of $\Xi$ will be the identity function: $\xi^{X}(n) = n$. 

			Since $n_{0} = 0$, $Y\restrict n_{0}$ is the empty string. Suppose $s\ge 0$ and $Y_{\restrict n_{s}}$ has already been defined.
			The set of ``good'' strings now is
			\[
				\mc G = \{\tau\succ Y_{\restrict n_{s}} \mid \Phi^{\tau\restrict n_{s + 1}}(i_{s + 1}) = 0\}.
			\]
			Define the ``cost'' of $\tau$ to be the additional Hamming distance to $X$, i.e., \[d(\tau) = \abs{(X + \tau)\cap [n_{s},n_{s + 1})}.\]
			\begin{description}
			\item[Case 1] $\mc G\ne\nil$. Then 
			let $Y_{\restrict n_{s + 1}}$ be any $\tau_{0}\in\mc G$ of length $n_{s + 1}$ and of minimal cost, i.e., such that 
			$d(\tau_{0}) = \min\{d(\tau)\mid \tau\in\mc G\}$. 
			That is, let 
			\[
				Y_{\restrict n_{s + 1}}\in\argmin_{\tau\in\mc G} d(\tau).
			\]

			\item[Case 2] Otherwise. Then make no further changes to $X$ up to length $n_{s + 1}$, i.e., 
			let $Y_{\restrict n_{s + 1}} = Y_{\restrict n_{s}}\searrow X_{\restrict n_{s + 1}}$.
			\end{description}
			This completes the definition of $\Xi$ and hence of $Y$. 
			It remains to show that $\Phi^{Y}\not\in\MLR$. For any string $\sigma$ of length $n_s$ let 
			\[
				E_{s + 1}^{\sigma} = \left\{Z\in \{0,1\}^{[n_s,n_{s+1})}: 
					\neg\left(\Phi^{\sigma\searrow Z}(i_{s + 1})\downarrow = 0@n_{s + 1}\right)				
				\right\}
			\]
			\[
				 = \left\{Z: 
					\Phi^{\sigma\searrow Z}(i_{s + 1})\downarrow = 1@n_{s + 1}
				\right\}
			\]
			\[
				\cup
				\left\{Z: \Phi^{\sigma\searrow Z}(i_{s + 1})\uparrow@n_{s + 1}\right\}
			\]
			Since (\ref{united}) and (\ref{qq}) hold for all strings of length $n_{s}$, in particular they hold for $\sigma = \Xi^{X}\restrict n_{s}$, so
			\begin{equation}\label{mash}
				(\forall s)\quad 
				\P(E^\sigma_{s + 1})\le \mf p + \frac1{2s} + \frac1{2s} 
				= \mf p +  \frac{1}{s},\quad\text{hence}\quad\limsup_{s\to\infty}\,\P(E^{X\restrict n_s}_{s + 1})
				\le\mf p.
			\end{equation}
			Let
			\[
				U_s^{X\restrict n_s} = \{Z\in\{0,1\}^{[n_s,n_{s+1})}:B^{[n_s,n_{s+1})}_{p(n_{s + 1}-n_{s})}(Z)\subseteq E^{X\restrict n_s}_{s + 1}\}.
			\]
			Since
			\[
				\frac{p(n_{s + 1}-n_{s})}{\sqrt{n_{s + 1}-n_{s}}}\to^{*}\infty.
			\]
			we can apply Lemma \ref{core} and there is $h(s)$ with $\limstar_{s\to\infty} h(s)=0$ and 
			\[
				\P(U_{s}^{X\restrict n_{s}}) \le h(s)
			\]
			 that only depends on 
			an upper bound for an $s_{0}$ such that for all $s\ge s_{0}$, $\P(E_{s + 1})\le \mf q$ 
			(where $\mf p<\mf q<1$ and $\mf q$ is just some fixed computable number). 
			Since by (\ref{mash}) such an upper bound can be given that works for all $X$, actually $h(s)$ may be chosen to not depend on $X$. 
			Let
			\[
				V_{s} = \{Z : Z\in U_{s}^{Z\restrict n_{s}}\},
			\]
			then $V_{s}$ is uniformly $\Delta^{0}_{1}$. To find the probability of $V_{s}$
			we note that for each of the $2^{n_s}$ possible beginnings of $Z$, there are at most $(h(s)\cdot 2^{n_{s+1}-n_s})$ continuations of $Z$ on $[n_s,n_{s+1})$ that make $Z\in V_s$; so we compute
			\[
				\P(V_s) = |\{Z\in \{0,1\}^{n_{s+1}} : Z\restrict [n_s,n_{s+1}) \in U^{Z\restrict n_s}\}| 2^{-n_{s+1}}
			\]
			\[
				\le 2^{n_s} (h(s)\cdot 2^{n_{s+1}-n_s}) 2^{-n_{s+1}} = h(s)
			\]
			so since $\limstar_{s\to\infty}h(s) = 0$, $\{V_{s}\}_{s\in\omega}$ is a Kurtz randomness test. 
			Let $\{m_{s}\}_{s\in\omega}$ be a computable sequence such that $\sum_{s\ge t} h(m_{s})\le 2^{-t}$. 
			Let $W_{t} = \bigcup_{s\ge t} V_{m_{s}}$. Then $\P(W_{t})\le 2^{-t}$ and $W_{t}$ is uniformly $\Sigma^{0}_{1}$ and hence it is a Martin-L\"of randomness test. 
			Since $X\in\MLR$, $X\not\in W_{t}$ for some $t$ and hence $X\not\in V_{m_{s}}$ for all but finitely many $s$. 
			So $\Phi^{Y}(m_{s}) = 0$ for all but finitely many $s$, hence $\Phi^{Y}\not\in\CIM$.

			By construction, we have
			\[
				\abs{(X + Y)\cap [n_{s},n_{s + 1})}\le p(n_{s + 1}-n_{s})
			\]
			for all but finitely many $n$.
			Therefore
			\[
				\abs{(X + Y)\cap [0,n_{s + 1})}\le \sum_{k = 0}^{s}p(n_{k + 1}-n_{k})\le p(n_{s + 1})
			\]
			so $X\sim_{p,N}Y$ where $N = \{n_{s}:s\in\omega\}$.
		\end{proof}

		\subsection{Main result}

		\begin{lem}\label{K}
			Let $p(n)=o(\frac{n}{\log n})$ and let $N\in\mf C$. 
			If $X\in\MLR$ and $X\sim_{p,N}Y$ then $\dim_{cp}(Y)=1$.
		\end{lem}
		\proof
			Suppose there are at most $p(n)$ many bits changed to go from $X\restrict n$ to $Y\restrict n$, in positions $a_{1},\ldots,a_{p(n)}$.
			(In case there are fewer than $p(n)$ changed bits, we can repeat $a_{i}$ representing the bit $0$ which we may assume is changed.)
			Let $(Y\restrict n)^{*}$ be a shortest description of $Y\restrict n$. From the code
			\[
				0^{\abs{K(Y\restrict n)}}{^{\frown}1}{^{\frown}}K(Y\restrict n)^{\frown}(Y\restrict n)^{*\frown} a_{1}\cdots a_{p(n)}
			\]
			we can effectively recover $X\restrict n$. Thus
			\[
				n-c_{1}\le K(X\restrict n)\le 2\log [K(Y\restrict n)] + 1 + K(Y\restrict n) + p(n)\log n + c_{2}
			\]
			\[
				\le 2\log [n + 2\log n + c_{3}] + 1 + K(Y\restrict n) + p(n)\log n + c_{2}.
			\]
			Hence
			\[
				n\le^{ + } 3\log n + K(Y\restrict n) + p(n)\log n,\quad\text{and}
			\]
			\[
				n-(p(n) + 3)\log n \le^{ + } K(Y\restrict n).\eqno{\qEd}
			\]

		\begin{thm}\label{notsur}
			For each Turing reduction procedure $\Phi$ 
			there is a set $Y$ with $\dim_{cp}(Y) = 1$ such that $\Phi^{Y}$ is not stochastically bi-immune. 
		\end{thm}
		\begin{proof}
			Let $p(n) = n^{2/3}$, so that $p(n) = o(n/\log n)$ and $p(n) = \omega^{*}(\sqrt{n})$. 
			By the proof of Theorem \ref{Main} and since the sequence of numbers $n_{s}$ is computable, 
			for each weakly 3-random set $X$ there is a set $Y\sim_{p, N}X$ (for some $N\in\mf C$) such that 
			$\Phi^{Y}$ is not both co-immune and in $\SD_{1/2}$, in particular $\Phi^{Y}\not\in\SBI$. 
			By Lemma \ref{K}, each such $Y$ has complex packing dimension 1.
		\end{proof}

	
\end{document}